\documentclass[11pt,reqno]{amsart}
\usepackage{amsmath}
\usepackage{amsthm}
\usepackage{latexsym}
\usepackage{graphicx}
\usepackage{amssymb}
\usepackage[latin1]{inputenc}

\theoremstyle{plain}

\numberwithin{equation}{section}
\newtheorem{thm}{Theorem}[section]
\newtheorem{theorem}[thm]{Theorem}

\newtheorem{definition}[thm]{Definition}
\newtheorem{proposition}[thm]{Proposition}
\newtheorem{corollary}[thm]{Corollary}
\newtheorem{remark}[thm]{Remark}

\begin{document}

\setcounter{page}{1}
\title[Solving the Pell equation via Rédei rational functions]{Solving the Pell equation via Rédei rational functions}
\author{Stefano Barbero}
\address{Dipartimento di Matematica\\
                Universit\`{a} di Torino\\
                Via Carlo Alberto 8\\
                Torino, Italy}
\email{stefano.barbero@unito.it}
\author{Umberto Cerruti}
\address{Dipartimento di Matematica\\
                Universit\`{a} di Torino\\
                Via Carlo Alberto 8\\
                Torino, Italy}
\email{umberto.cerruti@unito.it}
 \author{Nadir Murru}
\address{Dipartimento di Matematica\\
                Universit\`{a} di Torino\\
                Via Carlo Alberto 8\\
                Torino, Italy}
\email{nadir.murru@unito.it}

\begin{abstract}
In this paper, we define a new product over $\mathbb{R}^{\infty}$, which allows us to obtain a group isomorphic to $\mathbb R^*$ with the usual product. This operation unexpectedly offers an interpretation  of the Rédei rational functions, making more clear some of their properties, and leads to another product, which generates a  group structure over the Pell hyperbola. Finally, we join together these results, in order to evaluate solutions of Pell equation in an original way.
\end{abstract}

\maketitle

\section{Introduction}

The Pell equation is
$$x^2-dy^2=1,$$
where $d$ is a given positive integer and $x$, $y$ are unknown numbers, whose values we are seeking over the integers. If $d$ is a perfect square, the Pell equation has only the trivial solution $(1,0)$. The interesting case is when $d$ is not a square, because then the Pell equation has infinite integer solutions. For example, knowing the minimal non trivial solution $(x_1,y_1)$ with $x_1$ and $y_1$ positive integers, we can evaluate
$$(x_1+y_1\sqrt{d})^n=x_n+y_n\sqrt{d}, \quad \forall n\geq0$$
where $(x_n,y_n)$ is also a solution of the Pell equation (cf. \cite{Jac}). The Pell equation has ancient origins (for further references, see \cite{Dick}, \cite{Weil}) dating back to Archimedes Cattle Problem, but it is not known if Archimedes was able to solve it. Surely the first mathematician who found a method for solving the Pell equation was the Indian  Brahmagupta. If you know two solutions $(x_1,y_1)$ and $(x_2,y_2)$ of the Pell equation, with the Brahmaguptha method you find that
$$(x_1x_2+dy_1y_2,x_1y_2+y_1x_2),\quad (x_1x_2-dy_1y_2, x_1y_2+y_1x_2)$$ 
are also solutions. Nowadays the most common method of resolution involves continued fractions. A continued fraction is a representation of a real number $\alpha$ through a sequence of integers as follows:
$$\alpha=a_0+\cfrac{1}{a_1+\cfrac{1}{a_2+\cfrac{1}{a_3+\cdots}}}\ ,$$
where the integers $a_0,a_1,...$ can be evaluated with the recurrence relations 
$$\begin{cases} a_k=[\alpha_k]\cr \alpha_{k+1}=\cfrac{1}{\alpha_k-a_k} \quad \text{if} \ \alpha_k \ \text{is not an integer}  \end{cases}  \quad k=0,1,2,...$$
for $\alpha_0=\alpha$ (cf. \cite{Olds}). A continued fraction can be expressed in a compact way using the notation $[a_0,a_1,a_2,a_3,...]$. The finite continued fraction
$$[a_0,...,a_n]=\cfrac{p_n}{q_n}\ ,\quad n=0,1,2,...$$
is a rational number and is called the $n$--th \emph{convergent} of $[a_0,a_1,a_2,a_3,...]$. The continued fraction expansion of $\sqrt{d}$ is periodic and it has the form (cf. \cite{Olds})
$$\sqrt{d}=[a_0,\overline{a_1,...,a_{L-1},2a_0}]\ .$$
It is possible to prove that the minimal solution of the Pell equation is
$$(x_1,y_1)=(p_{L-1},q_{L-1}),\quad L \ \hbox{even}$$
$$(x_1,y_1)=(p_{2L-1},q_{2L-1}),\quad L \ \hbox{odd},$$
where
$$\cfrac{p_{L-1}}{q_{L-1}}=[a_0,...,a_{L-1}]$$
$$\cfrac{p_{2L-1}}{q_{2L-1}}=[a_0,...,a_{2L-1}]\ .$$
Furthermore all solutions of the Pell equation are of the form $(p_{nL-1},q_{nL-1})$, when $L$ is even, and $(p_{2nL-1},q_{2nL-1})$ when $L$ is odd, for $n=1,2,...$( cf. \cite{Olds}).\\
The geometrical locus containing  all solutions of  the Pell equation is the Pell hyperbola
$$H_d=\{(x,y)\in\mathbb R^2:x^2-dy^2=1\}\ .$$
The fascinating aspect of working with the Pell hyperbola is a group law that we can give over $H_d$. In fact, for every pair of points $P,Q\in H_d$, we can define their product as the intersection between $H_d$ and a line through $(1,0)$ parallel to the line $PQ$. This geometric construction is essentially the same presented in the book of Veblen (cf. \cite{Veblen}) as a general product between points on a line (\cite{Veblen} p. 144), which can be defined in a similar way over conics (\cite{Veblen}, p.231--232), in order to obtain a group structure. Algebraically (see, e.g., \cite{Jac}) this product between two points $(x_1,y_1)$ and $(x_2, y_2)$ is given by the point
$$(x_1x_2+dy_1y_2, y_1x_2+x_1y_2)\ .$$
This geometric group law over the Pell hyperbola, similar to that one over an elliptic curve, has the advantage of a nice algebraic expression, which yields greatly simplified formulas (cf. \cite{Lem}). In this paper we will see how it is possible to obtain this
product in an original way, starting essentially from a transform over $\mathbb R \cup \{\infty\}$. Then we will use the Rédei rational functions (cf. \cite{Redei}) in order to evaluate the powers of points over $H_d$.\\
The Rédei rational functions arise from the development of $(z+\sqrt{d})^n$, where $z$ is an integer and where $d$ is a nonsquare positive integer. One can write
\begin{equation}\label{pow}(z+\sqrt{d})^n=N_n(d,z)+D_n(d,z)\sqrt{d}\ ,\end{equation}
where
$$ N_n(d,z)=\sum_{k=0}^{[n/2]}\binom{n}{2k}d^kz^{n-2k}, \quad D_n(d,z)=\sum_{k=0}^{[n/2]}\binom{n}{2k+1}d^kz^{n-2k-1}.  $$
The Rédei rational functions $Q_n(d,z)$ are defined by 
\begin{equation}\label{red}Q_n(d,z)=\cfrac{N_n(d,z)}{D_n(d,z)}, \quad \forall n\geq1 \ .\end{equation}
It is well--known their multiplicative property
\begin{equation*}Q_{nm}(d,z)=Q_n(d,Q_m(d,z))\ ,\end{equation*}
for any couple of indexes $n,m$. Thus the Rédei functions are closed with respect to composition and satisfy the commutative property
$$Q_n(d,Q_m(d,z))=Q_m(d,Q_n(d,z))\ .$$ 
The multiplicative property of the Rédei functions has several applications. For example, this property has been exploited  in order to create a public key cryptographic system (cf. \cite{Nob}). The Rédei rational functions reveal their utility in several other fields. Given a finite field $\mathbb F_q$, of order $q$, and $\sqrt{d}\not\in\mathbb F_q$, then $Q_n(d,z)$ is a permutation of $\mathbb F_q$ if and only if $(n,q+1)=1$ (see \cite{Lidl}, p. 44). Another recent application of these functions is in finding a new bound for multiplicative character sums of nonlinear recurring sequences (cf. \cite{Gomez}). Moreover, they can be used in order to generate pseudorandom sequences (cf. \cite{Topu}). In this paper we will see a totally different approach for the Rédei rational functions, studying an original connection between them and continued fractions. In particular we will see how we can apply Rédei rational functions in order to generate solutions of the Pell equation. This research starts from the fact that the sequence of functions $Q_n(d,z)$, by definition, converges to $\sqrt{d}$, i.e., Rédei functions are rational approximations of $\sqrt{d}$, for any parameter $z$. We know that the best rational approximations of an irrational number are provided by the convergents of its continued fraction (cf. \cite{Olds}). So it would be beautiful to find a parameter $z$ in order to obtain some convergents from $Q_n(d,z)$. In this paper we find the parameter $z$ such that $Q_{2n}(d,z)$ correspond to all the convergents of the continued fraction of $\sqrt{d}$, leading to the solutions of the Pell equation. This method uses the Rédei rational functions in a totally different field from the classic ones and, considering the fast evaluation of Rédei rational functions (cf. \cite{More}), it allows to generate solutions of the Pell equation in a rapid way.

\section{A new operation over $\mathbb R^{\infty}$ }

We consider a transform, defined over $\mathbb{R}^{\infty}=\mathbb{R}\cup\{\infty\}$ which induces a natural product over $\mathbb R^{\infty}$. We show how this operation is strictly related to the Rédei rational functions (cf. \cite{Redei}), revealing a new interesting point of view to understand their multiplicative property.
\begin{definition}\label{defrodi}
For any positive real number $d$,  we define the transform $$\rho_d:\mathbb{R}^{\infty}\rightarrow\mathbb{R}^{\infty}$$ 
\begin{equation}\label{rodi} \rho_d(x)=\cfrac{x+1}{x-1}\sqrt{d}\ . \end{equation}
\end{definition}
As immediate consequences of (\ref{rodi}), we have that
\begin{equation}\label{roval}\rho_d(1)=\infty, \quad \rho_d(0)=-\sqrt{d}, \quad \rho_d(\infty)=\sqrt{d}\ . \end{equation}
Moreover, we can easily find the inverse  of $\rho_d$:
\begin{equation}\label{roinv}\rho_d^{-1}(x)=\cfrac{x+\sqrt{d}}{x-\sqrt{d}}\ . \end{equation}
A more important achievement  is the product $\odot_{d}$, induced by $\rho_d$ over $\mathbb{R}^{\infty}$
\begin{equation} \label{prod} x \odot_{d} y=\rho_d(\rho_d^{-1}(x)\rho_d^{-1}(y))=\cfrac{d+xy}{x+y}\ , \end{equation}
which is surely associative and commutative. Furthermore, comparing $\odot_{d}$ with the usual product over $\mathbb R$,  $\sqrt{d}$ plays the role of $\infty$ and $-\sqrt{d}$ the role of 0, since we have the following relations for all $x\in\mathbb{R}^{\infty}$ 
$$ \sqrt{d}\odot_d x=\cfrac{d+\sqrt{d}x}{\sqrt{d}+x}=\sqrt{d} \quad \hbox{and} \quad -\sqrt{d}\odot_d x=-\sqrt{d}\ .$$
These properties, together with (\ref{roval}), suggest that this new product induces a group structure on the set $$P_d:=\mathbb R^\infty-\{\pm\sqrt{d}\}\ .$$ Over $(P_d,\odot_d)$, $\infty$ is the identity with respect to $\odot_d$ and any element $x$ has the unique inverse $-x$ since
$$ x \odot_d \infty=x,  \quad x \odot_d (-x)=\infty, \quad \forall x\in P_d \ . $$
We immediately have the following. 
\begin{theorem}
$(P_d,\odot_d)$ is a commutative group. The transform $\rho_d $ is an isomorphism from the group $(R^*,\cdot)$, of nonzero real numbers under ordinary multiplication, into the group $(P_d,\odot_d)$.
\end{theorem}
\begin{remark}
If we consider two positive real numbers $d$, $e$, we also have an immediate isomorphism between $(P_d,\odot_d)$ and $(P_e,\odot_e)$:
 $$\phi:P_e\rightarrow P_d$$ $$\phi:x \mapsto x\sqrt{\frac{d}{e}}\ .$$
\end{remark}
We now show a wonderful relation relating the product (\ref{prod}) and the Rédei rational functions (cf. \cite{Redei}). It is well--known that for any pair of indices $m$ and $n$,
\begin{equation}\label{mult}Q_{nm}(d,z)=Q_n(d,Q_m(d,z))\end{equation}
but what can we say about $Q_{n+m}(d,z)$? Using product (\ref{prod}) the answer is simple and furthermore we can use this fact in order to obtain the multiplicative property in a more explicative way than the usual one (see \cite{Lidl}, p. 22--23). 
\begin{proposition} \label{sum-prop}
With the notation introduced above
$$Q_{n+m}(d,z)=Q_n(d,z)\odot_d Q_m(d,z).$$
\end{proposition}
\begin{proof}
For the sake of simplicity, here we omit the variables $z$ and $d$, so $$N_n(d,z)=N_n,\quad D_n(d,z)=D_n,\quad Q_n(d,z)=Q_n .$$ 
Using a matricial approach (cf. \cite{Von}) we have 
\begin{equation} \label{redei-matrix} \begin{pmatrix} z & d \cr 1 & z  \end{pmatrix}^n=\begin{pmatrix} N_n & dD_n \cr D_n & N_n  \end{pmatrix}, \end{equation}
so
$$ \begin{pmatrix} N_{n+m} & dD_{n+m} \cr D_{n+m} & N_{n+m}  \end{pmatrix}=\begin{pmatrix} N_n & dD_n \cr D_n & N_n  \end{pmatrix}\begin{pmatrix} N_m & dD_m \cr D_m & N_m  \end{pmatrix}. $$
Comparing the resulting matrix on the right with the one on the left, we finally obtain the relations 
\begin{equation} \label{redei-index-sum} \begin{cases} N_{n+m}=N_nN_m+dD_nD_m \cr D_{n+m}=D_nN_m+N_nD_m \ .\end{cases} \end{equation}
Now the proof is straightforward using (\ref{redei-index-sum}) and definition of Rédei rational functions (\ref{red}):
$$Q_n\odot_d Q_m=\cfrac{d+Q_nQ_m}{Q_n+Q_m}=\cfrac{d+\frac{N_n}{D_n}\cdot\frac{N_m}{D_m}}{\frac{N_n}{D_n}+\frac{N_m}{D_m}}=\cfrac{N_nN_m+dD_nD_m}{N_nD_m+N_mD_n}=Q_{n+m}\ . $$
\end{proof}
\begin{remark}
It is interesting to consider the additive property of the Rédei rational functions because it clarifies the meaning of the multiplicative property and makes this relation simpler to use. Furthermore, we can relate these functions to the integers with the operations of sum and product, extending their definition to negative indexes $n$ through this matricial approach. Indeed the set of the Rédei rational functions $\{Q_n(d,z), \forall n\in \mathbb Z\}$ with the operations $\odot_d$ and $\circ$ (the composition of functions) is a ring isomorphic to $(\mathbb Z,+,\cdot)$.
\end{remark}
The previous proposition is a powerful tool, which enables us to evaluate the powers of elements with respect to the product (\ref{prod}).
\begin{corollary} \label{zn}
Let $z^{n_{\odot_d}}=\underbrace{z\odot_d \cdots \odot_d z}_n$ be the $n$--th power of $z$ with respect to the product (\ref{prod}). Then
$$z^{n_{\odot_d}}=Q_n(d,z)\ .$$
\end{corollary}
\begin{proof}
From (\ref{pow}) and (\ref{red}), when $n=1$
$$z=Q_1(d,z)\ .$$
Therefore 
$$z^{n_{\odot_d}}=z\odot_d \cdots \odot_d z=Q_1(d,z)\odot_d\cdots \odot_d Q_1(d,z)= Q_{1+\cdots+1}(d,z)=Q_n(d,z).$$
\end{proof} 
This Corollary shines a light on the multiplicative property (\ref{mult}), which holds because $Q_n(d,z)$ is essentially a power of an element with respect to the product (\ref{prod}):
$$Q_n(d,Q_m(d,z))=(Q_m(d,z))^{n_{\odot_d}}=(z^{m_{\odot_d}})^{n_{\odot_d}}=z^{nm_{\odot_d}}=Q_{nm}(d,z).$$
This proof for the multiplicative property allows us to consider the Rédei rational functions as a power of an element and in this way they are easier to use, as we will see later.

\section{A commutative group over Pell hyperbola}
The group $(P_d,\odot_d)$, related to a nonsquare positive integer $d$, reveals some important connections to the Pell hyperbola
$$ H_d=\{(x,y)\in \mathbb R^2 / x^2-dy^2=1\},$$
with asymptotes $$y=\pm\cfrac{x}{\sqrt{d}}\ .$$ 
The Pell hyperbola $H_d$ can be parametrically represented using the line
\begin{equation} \label{line} y=\cfrac{1}{m}(x+1)\ . \end{equation}
Here we use, as parameter $m$, the cotangent of the angle formed with the $x$--axis, instead of the usual tangent. We choose the cotangent as parameter because it is the only choice which allows us to have a correspondence between $P_d$ and $H_d$. Indeed in this case $\infty$ (the identity of $P_d$) will correspond to $(1,0)$ (the identity of $H_d$). Moreover, the two points at the infinity are obtained by the parameters $\pm\sqrt{d}$.\\ 
Now, for all $m \not=\pm\sqrt{d}$, the line (\ref{line}) intersects $H_d$ in only two points: $(-1,0)$ and another one. Thus it seems natural to visualize $P_d$ as the parametric line. Moreover, the bijection 
$$ \epsilon_d:P_d\rightarrow H_d $$
$$\epsilon_d:m \mapsto \left(\cfrac{m^2+d}{m^2-d}\ , \cfrac{2m}{m^2-d}\right),$$
yields a solution to the system
\begin{equation} \label{tau} \begin{cases} x^2-dy^2=1 \cr y=\cfrac{1}{m}(x+1)\ . \end{cases}\end{equation}
In fact, we find two solutions $(-1,0),\ \left(\frac{m^2+d}{m^2-d}\ ,\frac{2m}{m^2-d}\right).$ The bijection $\epsilon_d$ maps 0 into $(-1,0)$, $\infty$ into $(1,0)$, and, using (\ref{line}), the inverse of $\epsilon_d$ can be seen to be
$$ \tau:H_d\rightarrow P_d $$
$$ \tau:(x,y)\mapsto \cfrac{1+x}{y}\ . $$
We observe that the map $\tau$ has no explicit dependence on $d$, but this dependence is implicit when we consider a point $(x,y)\in H_d$. Since we have a bijection from $P_d$ to $H_d$, we can transform the product $\odot_d$ over $P_d$ into a product $\odot_H$ over $H_d$ as 
$$ (s,t)\odot_H (u,v)=\epsilon_d(\tau(s,t)\odot_d\tau(u,v)), \quad \forall (s,t),(u,v)\in H_d \ . $$
Using the definitions of $\epsilon$ and $\tau$:
$$\tau(s,t)\odot_d\tau(u,v)=\left(\cfrac{1+s}{t}\right)\odot_d\left(\cfrac{1+u}{v}\right)=\cfrac{1+s+u+su+dtv}{t+tu+v+sv}\ ,$$
and, using the relations $s^2-dt^2=1$ and $u^2-dv^2=1$, finally we have
\begin{equation} \label{prodH} (s,t)\odot_H (u,v)=(su+dtv,tu+sv). \end{equation}
The operation constructed on $H_d$ has all the good properties of $\odot_d$. Therefore we have the following
\begin{theorem}
$(H_d,\odot_H)$ is isomorphic to $(P_d,\odot_d)$ and $(H_d,\odot_H)$ is a commutative group.
\end{theorem}
As we have seen in the introduction, the product (\ref{prodH}) is a classical product used to construct a group over a conic, but here we have seen how it is possible to obtain this product in an original way, starting essentially from a transform over $\mathbb R^{\infty}$.
\begin{remark}
$H_d$ is isomorphic to the group of unit norm elements in $\mathbb Q(\sqrt{d})$. Remembering that
$$\mathbb Q(\sqrt{d})=\{s+t\sqrt{d}:s,t\in\mathbb Q\},$$
the natural product between two elements of $\mathbb Q(\sqrt{d})$ is
$$(s+t\sqrt{d})(u+v\sqrt{d})=(su+dtv)+(tu+sv)\sqrt{d}.$$
\end{remark}
\begin{remark}
We can easily observe that the products $\odot_d$ and $\odot_H$, and the transformations $\epsilon_d, \tau$, all involve only rational operations. If we consider $\mathbb P_d=\mathbb Q \cup \{\infty\}=\mathbb Q^\infty$ and $\mathbb H_d=\{(x,y)\in H_d/x,y\in \mathbb Q\}$, then we retrieve again  $\epsilon_d, \tau$, as isomorphisms between $\mathbb P_d$ and $\mathbb H_d$. Furthermore $\mathbb P_d$ does not depend on $d$ and $(\mathbb P_d,\odot_d)\cong(\mathbb Q^*,\cdot)$. 
\end{remark}

\section{Generating solutions of the Pell equation via Rédei rational functions}
We are ready to put together all the results obtained in the previous sections. The main purpose of this work is to reveal a beautiful connection among Redéi rational functions, the product $\odot_H$ and solutions of the Pell equation.\\
The matrix representation of $N_n(d,z), D_n(d,z)$, introduced in Proposition \ref{sum-prop}, allows us to consider
$$(N_n(d,z))_{n=0}^{+\infty}=\mathcal W(1,z,2z,z^2-d)$$
$$(D_n(d,z))_{n=0}^{+\infty}=\mathcal W(0,1,2z,z^2-d)\ .$$
Here $(a_n)_{n=0}^{+\infty}=\mathcal W(a,b,h,k)$ indicates the linear recurrent sequence of order 2, with initial conditions $a,b$ and characteristic polynomial $t^2-ht+k$, i.e.,
$$\begin{cases}  a_0=a \cr a_1=b \cr a_n=ha_{n-1}-ka_{n-2} \quad \forall n\geq 2 \ . \end{cases}$$
Now for any point $(x,y)\in H_d$ we set $$(x_n,y_n)=(x,y)^{n_{\odot_H}},$$ where $$(x_0,y_0)=(1,0)\quad \text{and}\quad (x_1,y_1)=(x,y).$$ We know that $$(x+y\sqrt{d})^n=x_n+y_n\sqrt{d}\ ,$$ and so $(x_n), (y_n)$, as sequences, recur with polynomial $t^2-2xt+1$. We can easily observe that $(x,y)^{n_{\odot_H}}=(F_n(x),yG_n(x))$ and
$$(F_n(x))_{n=0}^{+\infty}=\mathcal W(1,x,2x,1)$$
$$(G_n(x))_{n=0}^{+\infty}=\mathcal W(0,1,2x,1)\ .$$
Now, comparing the recurrence of $F_n(x)$ and $N_n(d,z)$, we can see that they coincide when $$z=x \quad \hbox{and}\quad z^2-d=1\ ,$$
i.e., $d=z^2-1=x^2-1.$
We immediately have the following
\begin{proposition} With notation introduced above
$$F_n(x)=N_n(x^2-1,x)\quad \hbox{and}\quad G_n(x)=D_n(x^2-1,x).  $$
\end{proposition}
Let us recall Dickson polynomials (cf. \cite{Dick1}):
$$g_n(a,x)=\sum_{i=0}^{[n/2]}\cfrac{n}{n-i}\binom{n-i}{i}(-a)^ix^{n-2i},$$
introduced by L. E. Dickson over finite fields. He studied when they give permutation of the elements of the finite fields. Furthermore, Dickson polynomials are related to the classical Chebyshev polynomials and they are used in several areas, like cryptography, pseudoprimality testing, in order to construct irreducible polynomials over finite fields and in many other applications of the number theory. The Rédei rational functions are related to the Dickson polynomials:
$$2N_n(d,z)=g_n(2z,z^2-d).$$
When $a=1$, the sequence  $(g_n(1,x)=g_n(x))_{n=0}^{+\infty}$ recurs with characteristic polynomial $t^2-xt+1$ and we can observe that
$$F_n(x)=\cfrac{1}{2}g_n(2x)=N_n(x^2-1,x).$$
Since $(x,y)^{nm_{\odot_H}}=((x,y)^{n_{\odot_H}})^{m_{\odot_H}}$, $F_n(F_m(x))=F_{nm}(x)$ and we have retrieved the multiplicative property of $g_n(x)$. Furthermore we have the explicit formula for $F_n(x)$:
$$F_n(x)=\cfrac{1}{2}\sum_{i=0}^{[n/2]}\cfrac{n}{n-i}\binom{n-i}{i}(-1)^ix^{n-2i}\ .$$
Corollary \ref{zn} allows us to point out that 
$$Q_n(d,\cdot):(\mathbb Q^\infty,\odot_d) \rightarrow (\mathbb Q^\infty,\odot_d)$$
is a morphism for all $n$ and maps $z$ into $z^{n_{\odot_d}}$. In fact,
$$Q_n(d,a\odot_d b)=(a\odot_d b)^{n_{\odot_d}}=a^{n_{\odot_d}}\odot_d b^{n_{\odot_d}}=Q_n(d,a)\odot_d Q_n(d,b).$$
We have
$$Q_n\left(d,\cfrac{d+ab}{a+b}\right)=\cfrac{d+Q_n(d,a)Q_n(d,b)}{Q_n(d,a)+Q_n(d,b)}\ ,$$
and when $a=b=z$,
$$Q_n\left(d,\cfrac{d+z^2}{2z}\right)=\cfrac{d+Q_n(d,z)^2}{2Q_n(d,z)}=Q_n(d,z)\odot_d Q_n(d,z)=Q_{2n}(d,z).$$
\noindent Given $(x,y)\in H_d$, we obtain
$$(\tau(x,y))^{n_{\odot_d}}=\left(\cfrac{1+x}{y}\right)^{n_{\odot_d}}=Q_n\left(d,\cfrac{1+x}{y}\right).$$
Since $(\tau(x,y))^{n_{\odot_d}}=\tau((x,y)^{n_{\odot_H}})\ ,$ 
$$\cfrac{1+F_n(x)}{yG_n(x)}=\tau((F_n(x),yG_n(x)))=Q_n\left(d,\cfrac{1+x}{y}\right),$$
and thus
\begin{equation} \left( \cfrac{1+x}{y} \right)^{n_{\odot_d}}=\cfrac{1+N_n(x^2-1,x)}{yD_n(x^2-1,x)}\ .  \end{equation}
\begin{proposition}\label{prH}
For any $(x,y)\in H_d$
$$\left( \cfrac{1+x}{y} \right)^{2n_{\odot_d}}=\cfrac{F_n(x)}{yG_n(x)}=\cfrac{N_n(x^2-1,x)}{yD_n(x^2-1,x)}\ .$$
\end{proposition}
\begin{proof}
Here, we write $N_n, D_n$, instead of $N_n(x^2-1,x), D_n(x^2-1,x)$. We want to prove the equality
$$\cfrac{1+F_{2n}(x)}{yG_{2n}(x)}=\cfrac{F_n(x)}{yG_n(x)}\ ,$$
and so we consider
$$\cfrac{1+N_{2n}}{yD_{2n}}-\cfrac{N_n}{yD_n}=\cfrac{D_n+D_nN_{2n}-N_nD_{2n}}{yD_{2n}D_n}\ .$$
As a consequence of (\ref{redei-index-sum}) we have
$$N_{2n}(d,z)=N_n^2(d,z)+dD_n^2(d,z)$$
$$ D_{2n}(d,z)=2D_n(d,z)N_n(d,z)\ , $$
and
$$ \cfrac{1+N_{2n}}{yD_{2n}}-\cfrac{N_n}{yD_n}=\cfrac{D_n(1-N_n^2+(x^2-1)D_n^2)}{yD_{2n}D_n}=0. $$
In fact, using (\ref{redei-matrix}) 
$$N_n^2(d,z)-D_n^2(d,z)=(z^2-d)^n,$$
and in our case $$N_n^2(x^2-1,x)-(x^2-1)D_n^2(x^2-1,x)=1.$$
\end{proof}
\begin{corollary}
With the notation at the beginning of Section 3, for any $(x_1,y_1)\in H_d$, i.e., such that $x_1^2-dy_1^2=1$, we have
$$Q_{2n}\left( d,\cfrac{x_1+1}{y_1} \right)=\cfrac{x_n}{y_n}\ .$$
\end{corollary}
We want to apply all these tools, and in particular the Rédei rational functions, to the solutions of the Pell equation $x^2-dy^2=1$. In what follows, we will use Rédei rational functions to find convergents of continued fractions which provide solutions of the Pell equation.\\
Let $(x_1,y_1)$ be the minimal integer solution of the Pell equation and let $\cfrac{p_n}{q_n}\ ,$ for $n=0,1,2,...$, be the convergents of the continued fraction of $\sqrt{d}$ and $L$ the length of the period. As we have seen in the introduction, it is well--known that $(x_1,y_1)=(p_{L-1},q_{L-1})$ when $L$ is even and $(x_1,y_1)=(p_{2L-1},q_{2L-1})$ when $L$ is odd (see  \cite{Olds}). Furthermore, all the solutions of the Pell equation are $(p_{nL-1},q_{nL-1})$, when $L$ is even, and $(p_{2nL-1},q_{2nL-1})$, when $L$ is odd, for $n=1,2,...$\ . Then we have $$(x_1,y_1)^{n_{\odot_H}}=(x_n,y_n)=(p_{nL-1},q_{nL-1})$$ or $$(x_1,y_1)^{n_{\odot_H}}=(x_n,y_n)=(p_{2nL-1},q_{2nL-1})\ ,$$
when $L$ is even or odd respectively,  and by the last Corollary:
$$Q_{2n}\left(d,\cfrac{x_1+1}{y_1}\right)=\cfrac{p_{nL-1}}{q_{nL-1}}, \quad L \ \hbox{even} $$
$$Q_{2n}\left(d,\cfrac{x_1+1}{y_1}\right)=\cfrac{p_{2nL-1}}{q_{2nL-1}}, \quad L\ \hbox{odd} \ .$$ 
In other words, using Rédei rational functions, we can evaluate all the convergents of the continued fraction of $\sqrt{d}$ leading to the solutions of the Pell equation.

\medskip

\noindent AMS Classification Numbers: 11A55, 11B39, 11D09 

\end{document}